\documentclass[11pt]{article}
\usepackage{amsmath}
\numberwithin{equation}{section}
\usepackage[english]{babel}
\usepackage[utf8]{inputenc}
\usepackage{amssymb,amsmath,amsfonts,amsthm,latexsym,enumerate,url,cases}
\usepackage{graphicx}
\usepackage[colorinlistoftodos]{todonotes}
\usepackage[affil-it]{authblk}
\usepackage{hyperref}
\usepackage{indentfirst}
\usepackage{hyperref}
\usepackage{graphicx}
\usepackage{hyperref}
\usepackage{fullpage}
\usepackage{cite}
\allowdisplaybreaks[4]
\newtheorem{theorem}{Theorem}[section] %
\newtheorem{lemma}{Lemma}[section] %
\newtheorem{corollary}{Corollary}[section] %
\newtheorem{definition}{Definition}[section] %
\newtheorem{remark}{Remark}[section] %

\usepackage{stmaryrd}
\SetSymbolFont{stmry}{bold}{U}{stmry}{m}{n}

\usepackage{amsmath,amssymb}

\begin{document}
\title{Synchronization of Differential Equations Driven by Linear Multiplicative Fractional Brownian Motion}

\renewcommand{\thefootnote}{\alph{footnote}}
\author{Wei Wei$^{a}$,\ Hongjun Gao$^{b}$\thanks{Corresponding author, hjgao@seu.edu.cn}$\;$ and Qiyong Cao$^{b}$\\
{\small \it ${}^a$ School of Mathematical Sciences, Nanjing Normal University, Nanjing 210023, P. R. China}\\
{\small \it ${}^b$ School of Mathematics, Southeast University, Nanjing 211189, P. R. China}
}


\date{}
\maketitle
\begin{abstract}
This paper is devoted to the synchronization of stochastic differential equations driven by the linear multiplicative fractional Brownian motion with Hurst parameter $H\in(\frac{1}{2},1)$. We firstly prove that the equation has a unique stationary solution which generates a random dynamical system. Moreover the system has the pathwise singleton sets random attractor. Next we show up the synchronization of solutions of two coupled differential equations. At the end, we discuss two specific situations and provide the corresponding synchronization results.
\end{abstract}
\vspace{0.5cm}
	{\bf  The pathwise synchronization of stochastic differential equations driven by  multiplicative fractional Brownian motion(fBM) is very important problem. It is worth mentioning that most of the current research on the pathwise synchronization of stochastic differential equations with fBM focused on the additive situations, which has limitations in the real situations. Until now, the research for multiplicative fBm still has a big challenge even for linear multiplicative fBM. In this paper, we will consider the pathwise synchronization of stochastic differential equations driven by linear multiplicative fBM  with Hurst parameter $H\in(\frac{1}{2},1)$,  the reason we could obtain pathwise sychronization for SDE driven by linear multiplicative fBM due to we find a new transformation which we could transfer to SDE to random differential equation, and this transformation has chain rule and regularity which we need in the proof.}

\vspace{0.5cm}
\section{Introduction}
Synchronization is a fascinating phenomenon observed in various systems across diverse scientific disciplines, from physics and engineering to biology and economics. One of the first researchers who to study synchronization was Christiaan Huygens in the seventeenth century, when he noticed that two pendulum clocks mounted on the same frame would eventually synchronize and adjust their rhythms \cite{MR1898084}. It represents the remarkable ability of coupled dynamical systems to achieve a coordinated behavior, where the constituent components evolve in harmony over time. This emergent behavior has captured in imagination of scientists and researchers for decades, inspiring investigations into the underlying mechanisms and mathematical frameworks that govern synchronization. For different definitions and applications of synchronization we refer to the book \cite{MR1869044}.

In mathematical models, synchronization is often dealt with in asymptotic sense, and there have been quiet a few synchronization results for various types of differential equations. For ordinary differential equations (ODEs), the synchronization of coupled dissipative systems has been investigated in the case of autonomous systems by \cite{MR1612059,MR1623511} and nonautonomous systems by \cite{MR1639303,MR1971105,MR2241609}, provided that asymptotically stable equilibria and appropriate attractors are considered. In recent years, there has been a growing emphasis on considering the impact of stochastic factors on models and some crucial applications in relation to randommness. Consequently, there have been numerous research achievements in the field of sychronizations for stochastic differential equations within the framework of random attractors and stochastic stationary solutions in random dynamical systems. For Gaussian noise, Caraballo and Kloeden proved that synchronization in coupled determinstic disspative dynamical systems persists in the presence of environmental noise in \cite{MR2154449}. Regarding additive noise, there are also references to \cite{MR3639115,MR2286016}. Furthermore, \cite{MR3869887,MR2399931} have investigated linear multiplicative noise. Work by \cite{MR4026947} has dealt with nonlinear multiplicative Gaussian noise and yielded synchronization results in the mean square sense. Also \cite{MR2683802} provides synchronization results under additive L$\acute{e}$vy noise. Additionally, there are many conclusions in infinite-dimensional systems available in the book \cite{MR4181483}.

The fractional Brownian motion appears naturally in the modeling of many complex phenomena in applications when systems are subject to the external forcing which is neither a semmimartingale nor a Markov process. When the Hurst parameter $H$ is less than $\frac{1}{2}$, the properties of fractional Brownian motion become significantly more complex, and relevant calculates rely on rough path theory. As a result, current synchronization research is primarily focused on the case where $H$ is greater than $\frac{1}{2}$. \cite{MR3377401} presented conclusions regarding additive noise in finite lattice systems. However, there is currently a lack of research on linear multiplicative noise that fractional noise cannot be directly using the properties of Gaussian process like Ito's formula, etc. Nevertheless, by \cite{MR3377401}, we can justify the existence of fractional Orstein-Uhlenbeck processes. Therefore, inspired by \cite{MR3869887}, we attempt to velidate synchronization results for linear multiplicative fractional noise by using some techniques.

Mathematically, synchronization of two or more coupled systems means that the distance of the trajectories of the subsystems converge to together when the time goes to infinity. In this paper, we consider two stochastic differential equations driven by linear multiplicative noises :
\begin{align}
dX_t=f(X_t)dt+(a_1X_t+b_1)dB^{H_1}_t,  \notag \\
dY_t=g(Y_t)dt+(a_2Y_t+b_2)dB^{H_2}_t,
\end{align}
where $B^{H_1}_t,B^{H_2}_t$ are two independent $\mathbb{R}$-valued two-sided fractional Brownian motions with Hurst paramaters $H_1, H_2\in (\frac{1}{2},1)$. Drift term $f,g$ satisfy the following properties, here $\left\|\cdot\right\|$ denotes the Euclidean norm in $\mathbb{R}^d$:
\begin{enumerate}
    \item Linear growth condition, i.e. for all $x \in\mathbb{R}^d$,
\begin{align}
\left \| f(x) \right \|\leq l \left\| x \right \|+ M,\ \ \left \|g(x)  \right \| \leq l\left \| x \right \|+M,
\end{align}
where $l>0$ and $M>0$.
\item One-sided dissipative Lipschitz condition, i.e. for all $x,y \in\mathbb{R}^d$,
\begin{align}
   & \left \langle x-y,f(x)-f(y)\right\rangle\leq -L\left\| x-y \right \|^2,\notag \\
   & \left \langle x-y,g(x)-g(y)\right\rangle\leq -L\left\| x-y \right \|^2,
\end{align}
where $L>0$.
\item Integrability condition, i. e. there exists $m_0>0$ such that
\begin{align}
    \int^0_{-\infty}e^{ms}\left ( \left \| f(u(s)) \right \|^2+\left \| g(u(s)) \right \|^2 \right )ds<{\infty}
\end{align}
for any $m\in(0,m_0]$, and any continuous function $u:\mathbb{R}\rightarrow\mathbb{R}^d$ with uniform sub-exponential growth ( i.e. for any $\epsilon>0$ and $\omega\in\Omega$, there exists a $t_0(\epsilon,\omega)$, such that $\left \|u(t)\right \|\leq Ce^{\epsilon |t|}$ holds for $|t|\geq t_0(\epsilon,\omega)$ ).
\end{enumerate}

We want to acquire that synchronization occurs when we coupling two equations above and focus on the case that $a_1, a_2 \neq 0$. It is more convenient to deal with random differential equations compared to the stochastic differential equations. Hence our first step is to use the fractional Orstein-Uhlenbeck transformation
\begin{align}
\mathcal{T}_1: X_t \mapsto e^{-a_1O^1_t}(X_t+\frac{b_1}{a_1})-\frac{b_1}{a_1}=:U_t, \notag \\
\mathcal{T}_2: Y_t \mapsto e^{-a_2O^2_t}(X_t+\frac{b_2}{a_2})-\frac{b_2}{a_2}=:V_t
\end{align}
 to reduce system (1.1) to its equivalent system,
\begin{align}
    \frac{dU}{dt}=e^{-a_1O^1_t}f(U_te^{a_1O^1_t}+\frac{b_1}{a_1}(e^{a_1O^1_t}-1))+(a_1U_t+b_1)O^1_t,  \notag \\
    \frac{dV}{dt}=e^{-a_2O^2_t}g(V_te^{a_2O^2_t}+\frac{b_2}{a_2}(e^{a_2O^2_t}-1))+(a_2V_t+b_2)O^2_t,
\end{align}
where $O^1_t, O^2_t$ are fractional Orstein-Uhlenbeck processes satisfying
$$dO^1_t=-O^1_tdt+dB^{H_1}_t,\ \ dO^2_t=-O^2_tdt+dB^{H_2}_t.$$

However, to accomplish this step we need to note that the driving noises of the system are fractional Brownian motions, which are neither semmimartingales nor Markov processes. Moreover, the chain rule for the computation of fractional Ornstein-Uhlenbeck processes is currently missing. To overcome this difficulty, we adopt the Young integral \cite{MR4174393} for the definition of the stochastic integral since the paths of fractional Brownian motion have $\alpha$-H\"{o}lder continuity, where $\alpha$ is less than the Hurst parameter. Combined with the definition of the Young integral we verify that the chain rule still holds, and thus we can successfully apply the fractional O-U transformation. Before proceeding to the next step, we also establish the equivalence in conjugacy between (1.1) and (1.6). It is worth noting that in this paper, the space in which the systems reside is the H\"{o}lder continuous function space.

Next, we demonstrate the stability of the solutions for (1.6) and established the existence of random dynamical systems. Furthermore, the random dynamical system exhibit singleton attractors. Subsequently, we introduce linear cross-coupling to (1.5):
\begin{align}
    \left\{\begin{array}{l}
\frac{dU_{\kappa}}{dt}=e^{-a_1O^1_t}f(U_{\kappa}e^{a_1 O^1_t}+\frac{b_1}{a_1}(e^{a_1 O^1_t}-1))+(a_1 U_{\kappa}+b_1)O^1_t+\kappa(V_{\kappa}-U_{\kappa}),\\
\frac{dV_{\kappa}}{dt}=e^{-a_2O^2_t}g(V_{\kappa}e^{a_2 O^2_t}+\frac{b_2}{a_2}(e^{a_2 O^2_t}-1))+(a_2 V_{\kappa}+b_2)O^2_t+\kappa(U_{\kappa}-V_{\kappa}),
\end{array}\right.
\end{align}
and investigated the asymptotic behavior of solutions as the coupling strength parameter $\kappa$ tends to infinity. Indeed, we achieve synchronization of the system at the level of trajectories, i.e. $\lim_{\kappa\rightarrow\infty}\left \| U_{\kappa}(t,\omega)-V_{\kappa}(t,\omega) \right \|^2=0$. Furthermore, we demonstrate that as $\kappa$ approaches infinity, the solutions of (1.7) converge to the stationary solution of a so-called "averaged system":
\begin{align}
            \frac{dW}{dt}=&\frac{1}{2}[e^{-a_1O^1_t}f(e^{a_1O^1_t}W+\frac{b_1}{a_1}(e^{a_1O^1_t}-1))+e^{-a_2O^2_t}g(e^{a_2O^2_t}W+\frac{b_2}{a_2}(e^{a_2O^2_t}-1))] \notag \\
            &+\frac{1}{2}[(a_1W+b_1)O^1_t+(a_2W+b_2)O^2_t].
    \end{align}
Notably, the attractors of these systems are all singleton sets.

Finally, we revert to a coupled form of system (1.1) through the fractional O-U transformation:
\begin{align}
    \left\{\begin{array}{l}
dX_t=[f(X_t)+\kappa(e^{2\eta_t}Y_t-X_t)+\kappa(\frac{b_2}{a_2}(e^{2\eta_t}-e^{a_1O^1_t})-\frac{b_1}{a_1}(1-e^{a_1O^1_t}))]dt\\
\ \ \ \ \ \ \ \ +(a_1X_t+b_1)dB^{H_1}_t,\\
dY_t=[g(Y_t)+\kappa(e^{-2\eta_t}X_t-Y_t)+\kappa(\frac{b_1}{a_1}(e^{-2\eta_t}-e^{a_2O^2_t})-\frac{b_2}{a_2}(1-e^{a_2O^2_t}))]dt\\
\ \ \ \ \ \ \ \ +(a_2Y_t+b_2)dB^{H_2}_t,
\end{array}\right.
\end{align}
where $\eta_t=\frac{1}{2}(a_1O^1_t-a_2O^2_t)$, and we are able to obtain the synchronization results for this coupled system. At the end of the paper, we also explore two special cases as follows:

\begin{center}
\textit{Case 1: linear multiplicative noise}
\end{center}
\begin{align}
    \left\{\begin{matrix}
dX_t=f(X_t)dt+a_1X_tdB^{H_1}_t,\\
dY_t=g(Y_t)dt+a_2Y_tdB^{H_2}_t,
\end{matrix}\right.
    \end{align}

\begin{center}
\textit{Case 2: mixed noise}
\end{center}
\begin{align}
 \begin{cases}
     dX_t=f(X_t)dt+(aX_t+b_1)dB^{H_1}_t,\\
     dY_t=g(Y_t)dt+b_2dB^{H_2}_t,
 \end{cases}
\end{align}
and provide the corresponding synchronization results for each of them.

A brief outline of the paper is as follows. In Section 2, we review relevant concepts of random dynamical systems and stochastic integration, and introduce the fractional Ornstein-Uhlenbeck process. In Section 3, we present the main change of variables that will be used to transform our initial system (1.1) into system (1.5) and demonstrate the validity of this transformation along with the conjugate equivalence between SDEs and RDEs. In Section 4, we verify the stability of solutions and also prove the existence of the random attractors. Section 5 and Section 6 respectively devoted to the asymptotic behavior of systems and synchronization results. Further, synchronization results for two special cases were provided in Section 7.

\section{Preliminaries}
In this section, we will recall some concepts on random dynamical systems and stochastic integrals.

\subsection{Random Dynamical Systems}
We recall some basic theory in random dynamical systems in this subsection. More details and relevant proofs are referenced in \cite{MR1723992}.

Let $(\Omega, \mathcal{F}, \mathbb{P})$ be a probability space.

\begin{definition}
$(\Omega, \mathcal{F}, \mathbb{P}, \theta)$ is called a \emph{metric dynamical system} if:
\begin{enumerate}
    \item $\theta: \mathbb{R}\times\Omega \rightarrow \Omega$ is $(\mathcal{B}(\mathbb{R})\bigotimes\mathcal{F},\mathcal{F})$-measurable.
    \item $\theta_0\omega=\omega$ for all $\omega\in\Omega$.
    \item $\theta_{t+s}=\theta_t\circ\theta_s$ for all $t,s\in\mathbb{R}$.
    \item $\theta_t\mathbb{P}=\mathbb{P}\circ\theta_t^{-1}=\mathbb{P}$.
\end{enumerate}
\end{definition}

Moreover, the set $A\in\mathcal{F}$ is called a \emph{$\theta$-invariant set} if $\theta_tA=A$ for all $t\in\mathbb{R}$.
The metric dynamical system is called \emph{ergodic} if every $\theta$-invariant set in $\mathcal{F}$ has measure 0 or 1.
\begin{definition}
    Let $(\Omega, \mathcal{F}, \mathbb{P}, \theta)$ be a metric dynamical system.  A mapping $\phi: \mathbb{R}\times \Omega\times \mathbb{R}^d \rightarrow \mathbb{R}^d$ is called a random dynamical system if:
    \begin{enumerate}
        \item $\phi$ is $\left( \mathcal{B}(\mathbb{R}) \bigotimes \mathcal{F} \bigotimes \mathcal{B}(\mathbb{R}^d), \mathcal{B}(\mathbb{R}^d)\right)$-measurable.
        \item The mappings $\phi(t,\omega):=\phi(t,\omega,\cdot):\mathbb{R}^d \rightarrow \mathbb{R}^d $ form a \emph{cocycle} over $\theta(\cdot)$,  i.e.

    $\phi(0,\omega)=id_{\mathbb{R}^d}$ for all $\omega\in\Omega$; \\
    $\phi(t+s,\omega)=\phi(t,\theta_s\omega)\circ\phi(s,\omega)$ for all $s,t\in\mathbb{R}, \omega\in\Omega.$
    \end{enumerate}
\end{definition}
Here $ \circ$ means composition, which canonically defines an action on the left of the semigroup of self-mappings.

\emph{Random dynamical system(s)} is henceforth often abbreviated as $RDS$.
\begin{definition}
    A random variable $X\geq0$ is called \emph{tempered} if
    \begin{center}
        $\lim_{t\rightarrow\pm\infty}\frac{\log^+X(\theta_t\omega)}{|t|}=0$,\ \  for $\omega\in\Omega$ ,
    \end{center}
    where $\log^+x=\max\{0,\log x\}$ for $x\geq0$.
\end{definition}

\begin{definition}
    \begin{enumerate}
        \item A map $D:\Omega\rightarrow\mathcal{B}(\mathbb{R}^d)\setminus \{\varnothing\}$ is called a \emph{random set} if for each $x\in\mathbb{R}^d$ the map $\omega\mapsto d(x,D(\omega))$ is measurable, where $d$ is the \emph{Hausdorff semi-distance}, i.e. for $X,Y\subseteq \mathbb{R}^d$,
        \begin{center}
           $d(X,Y)=\sup_{x\in X}\inf_{y\in Y}\left \| x-y \right \|$.
        \end{center}
        \item A random set $D$ is called \emph{closed} if $D(\omega)$ is closed for all $\omega\in\Omega$.
        \item A random set $D$ is called \emph{tempered} if
        \begin{center}
            $\omega\mapsto\sup_{x\in D(\omega)}\left \| x \right \|$
        \end{center}
        is a tempered random variable.
        \item A random set $D$ is called \emph{forward invariant} with respect to a random dynamical system $\phi$ if for all $\omega\in\Omega$, and $t\geq0$,
        $$\phi(t,\omega,D(\omega)\subseteq D(\theta_t\omega).$$
    \end{enumerate}
\end{definition}

We denote by $\mathcal{D}$ the family of all tempered, closed random sets.
\begin{definition}
    A random set $B$ is called \emph{pullback absorbing set} in $\mathcal{D}$ if for all $D\in\mathcal{D}$ and for each $\omega\in\Omega$ there exists a time $t_D(\omega)>0$ such that for all $t>t_D(\omega)$,
    $$\phi(t,\theta_{-t}\omega,D(\theta_{-t}\omega))\subseteq B(\omega).$$
\end{definition}
\begin{definition}
    A random set $\mathcal{A}\in\mathcal{D}$ is called \emph{pullback attractor} if
    \begin{enumerate}
        \item $\mathcal{A}$ is compact random sets.
        \item $\mathcal{A}$ is strictly invariant, i.e. for all $t>0$ and every $\omega\in\Omega$ we have $\phi(t,\omega,\mathcal{A}(\omega))=\mathcal{A}(\theta_t\omega)$.
        \item $\mathcal{A}$ attracts all sets in $\mathcal{D}$ in the pullback sense, i.e. for all $D\in\mathcal{D}$ and for each $\omega\in\Omega$ we have
        $$\lim_{t\rightarrow\infty}d(\phi(t,\theta_{-t}\omega,D(\theta_{-t}\omega)),\mathcal{A}(\omega))=0.$$
    \end{enumerate}
\end{definition}
In order to obtain the existence of random attractors we often employ the following theorem which can be found in \cite{MR1723992}.
\begin{theorem}
    Let $(\theta, \phi)$ be a continuous RDS on $\Omega\times\mathbb{R}^d$. If there exists a family $\hat{B}=\{ B(\omega), \omega\in\Omega \}$ of non-empty measurable compact subsets $B(\omega)$ of $\mathbb{R}^d$ and for every $D(\omega)\in\mathcal{D}$, there exists $T_D(\omega)\geq0$ such that
    \begin{center}
        $\phi(t,\theta_{-t},D(\theta_{-t}\omega))\subset B(\omega), \forall t\geq T_D(\omega)$,
    \end{center}
    then the RDS $(\theta,\phi)$ has a random attractor $\hat{\mathcal{A}}= \{ \mathcal{A}(\omega), \omega\in\Omega \}$ with the component subsets defined for each $\omega\in\Omega$ by
    \begin{center}
        $\mathcal{A}(\omega) = \bigcap_{\tau>0}\overline{\bigcup_{t\geq\tau}\phi(t,\theta_t\omega)B(\theta_t\omega)}.$
    \end{center}
\end{theorem}
\begin{definition}
    A random variable $u_s\in\mathbb{R}^d$ is called a (random) stationary point for RDS $\phi$ if
    \begin{center}
        $\phi(t,\omega,u_s(\omega))=u_s(\theta_t\omega)$, for all $t\geq0$, $\omega\in\Omega$.
    \end{center}
\end{definition}
Since $\theta_t$ is a stationary shift for the probability measure $\mathbb{P}$, then the random process $t\mapsto u_s(\theta_t\omega)$ is a stationary process. In particular, if $\phi$ is generated by a random differential equation then $t\mapsto u_s(\theta_t\omega)$ gives a stationary solution.

Note that if the random attractor consists of singleton sets, i.e. $\mathcal{A}(\omega) = \{ \hat{u}(\omega) \}$ for some random variable $\hat{u}$ , then $\hat{u}_t(\omega):=\hat{u}(\theta_t\omega)$ is a stationary stochastic process.

Since this paper aims to discuss system (1.1) via system (1.5), the equivalence between these two systems is crucial. We now provide a lemma (\cite{MR2091122}, Lemma 2.2) as a foundation for our subsequent work.
\begin{lemma}
    Let $\phi$ be a random dynamical system and $\mathcal{T}:\Omega\times\mathbb{R}^d\rightarrow\mathbb{R}^d$ a mapping such that for each $\omega\in\Omega$ the map $x\mapsto\mathcal{T}(\omega,x)$ is a homeomorphism with inverse $\mathcal{T}^{-1}(\omega,\cdot)$ and the mappings $\mathcal{T}(\cdot,x),\mathcal{T}^{-1}(\cdot,x)$ are measurable for fixed $x\in\mathbb{R}^d$.

    Then $\psi$ defined by
    $$\psi(t,\omega,v):=\mathcal{T}^{-1}\Big(\theta_t\omega,\phi(t,\omega,\mathcal{T}(\omega,v))\Big)$$
    for $v\in\mathbb{R}^d$,$\omega\in\Omega$ and $t\in\mathbb{R}$ is again a random dynamical system. In this case, $\phi$ and $\psi$ are called \emph{conjugated}.

    Moreover, if $\phi$ has a random attractor $\mathcal{A}$, we obtain that $\psi$ has also a random attractor, given by $\mathcal{T}^{-1}\mathcal{A}$.
\end{lemma}

\subsection{Fractional Stochastic Integral}
Until now, there have been various distinct definitions of integrals involving fractional Brownian motion. Below, we will present the fractional stochastic integral definition employed in this paper.

Let $I$ be a compact interval on $\mathbb{R}$. Denote $C(I,V)$ and $C^{\alpha}(I,V)$ by the space of V-valued continuous function and the space of V-valued $\alpha$-H\"{o}lder continuous function space respectively, where V is a Banach space equipped with norm $|\cdot|
$ and  $\alpha\in(0,1)$. A norm on the $C^{\alpha}(I,V)$ is given by
\begin{equation}
    ||u||_{\alpha}=\sup_{t\in I}|u(t)|+|||u|||_{\alpha}
\end{equation}
with
\begin{equation}
    |||u|||_{\alpha}=\sup_{t,s\in I, t\neq s}\frac{|u(t)-u(s)|}{|t-s|^\alpha}.
\end{equation}

Suppose that $Y\in C^\alpha(I,V),X\in C^\beta(I,L(V))$. We now give a meaning to the expression $\int Y_tdX_t$. A natural approach would be try to define the integral as limit of Riemann-Stieljes sums, that is
\begin{eqnarray}\label{2.1}
    \int_{0}^{1}Y_tdX_t=\lim_{|\mathcal{P}|\to0}\sum_{[s,t]\in\mathcal{P}}Y_s(X_t-X_s),
\end{eqnarray}
where $\mathcal{P}$ denotes a partition of [0,1] and $|\mathcal{P}| $ denotes the length of the largest element of $\mathcal{P}$. From Young's theory, we know that such a sum converges if $\alpha+\beta>1$. As a sequence we can define the \emph{Young integral} like above. More details could be found in \cite{MR1555421}.

Now we recall some basic facts about fractional Brownian motion. For the specific theory and related properties, see \cite{MR2387368,MR2378138}.

Given $H\in(0,1)$, a continuous centered Gaussian process $B^H_t,t\in\mathbb{R}$ with the covariance function
\begin{equation}
    \mathbb{E}B^H_tB^H_s=\frac{1}{2}(|t|^{2H}+|s|^{2H}-|t-s|^{2H}), t,s\in\mathbb{R}
\end{equation}
is called a two-sided one-dimensional fractional Brownian motion (fBm), and $H$ is called the Husrt parameter. In this paper, we only consider the situation that $H\in(\frac{1}{2},1)$. Furthermore, by the Kolmogorov's theorem we know that $B^H_t$ has a continuous version, see \cite{MR1070361}. Hence we could consider the canonical version of a fBm as below.

Let $\Omega:=C_0(\mathbb{R},\mathbb{R})$ be the space of continuous functions on $\mathbb{R}$ with values in $\mathbb{R}$. Here and below the subscript zero indicates the functions take value zero at zero, equipped with the compact open topology. Let $\mathcal{F}$ be the associated Borel $\sigma$-algebra generated by $C_0(\mathbb{R},\mathbb{R})$, and let $\mathbb{P}$ be the distribution of $B^H_t$. We consider the Wiener shift given by
\begin{equation}
    \theta_t\omega(\cdot):=\omega(\cdot+t)-\omega(t), \ \ t\in\mathbb{R}
\end{equation}
for $\omega\in\Omega$. Thanks to \cite{MR2836654}, the quadruple $(\Omega, \mathcal{F},\mathbb{P},{\theta_t})$ is an ergodic metric dynamical system.

According to the Lemma 2.6 in \cite{MR2095071}, we note that the fBm has polynomial growth as $|t|\to\infty$. In fact, for every $\omega\in\Omega_0$, where $\Omega_0\in\mathcal{F}$ is a $(\theta_t)_{t\in\mathbb{R}}$-invariant full measure set, there exists a random constant $K(\omega)>0$ such that
\begin{equation}
    |B^H_t(\omega)|\leq K(\omega)(1+|t|^2)
\end{equation}
for all $t\in\mathbb{R}$. From now on, we could change $\Omega$ such that $\omega=0$ for $\omega\notin\Omega_0$.

Moreover, by the Kolmogorove continuity criterion, $B^H_t$ has a version, denoted by $\omega$, which is $\alpha$-H\"{o}lder continuous on any bounded interval $[T_1,T_2]\subset\mathbb{R}$ for $\frac{1}{2}<\alpha<H$. Let $\Omega_1\subset C_0(\mathbb{R},\mathbb{R})$ the set of functions, which are $\alpha$-H\"{o}lder continuous on any interval $[T_1,T_2]$.
\begin{lemma}[Lemma 15, \cite{MR3072986}]
    $\Omega_1\in\mathcal{B}(C_0(\mathbb{R},\mathbb{R}))$ and $\mathbb{P}(\Omega_1)=1$. In addition, $\Omega_1$ is $(\theta_t)_{t\in\mathbb{R}}$-invariant.
\end{lemma}

In the following sections, we consider the ergodic metric dynamical system introduced above restrict to the set $\Omega_1$. More precisely, let $\mathcal{F}_1$ be the Borel $\sigma$-subalgebra of $\mathcal{F}$ with respect to $\Omega_1$, let $\mathbb{P}_1$ be the restriction of $\mathbb{P}$ to $\mathcal{F}_1$, and $\theta$ represents the restriction of the Wiener shift to $\Omega_1\times\mathbb{R}$. Then $(\Omega_1, \mathcal{F}_1, \mathbb{P}_1, \theta)$ forms a metric dynamical system such that for every $A_1\in\mathcal{F}_1$ and $A\in\mathcal{F}$ with $A_1=A\cap\Omega_1$, we have that $\mathbb{P}_1(A_1)=\mathbb{P}(A)$ independent of the representation by $A$. In addition, the ergodicity of $(\Omega, \mathcal{F}, \mathbb{P}, \theta)$ is transferred to $(\Omega_1, \mathcal{F}_1, \mathbb{P}_1, \theta)$.


\subsection{Fractional Ornstein-Uhlenbeck Process}

The fractional Ornstein-Uhlenbeck process, abbreviated as fractional O-U process, can be constructed as the solution of the SDE
\begin{equation}
    dO^\nu_t=-\nu O^\nu_tdt+dB^H_t,\qquad  O^\nu_0=x,
\end{equation}
where $\nu>0$ and $B^H_t$ is a fBm with Hurst parameter $H>\frac{1}{2}$. We can easily see that the $O^{\nu,x}_t=e^{-\nu(t-t_0)}x+e^{-\nu t}\int^t_{t_0}e^{\nu s}dB^H(s),  t>t_0$, is the solution of (2.7). And the stationary solution is the process
\begin{equation}
    O^\nu_t=e^{-\nu t}\int^t_{-\infty}e^{\nu s}dB^H(s),
\end{equation}
see \cite{MR2524684}. $O^\nu_t$ is called a fractional O-U process, and due to the smoothness of the integrand, we conclude that this integral is well-defined. For simplicity, in this paper we let $\nu=1$, and denote $O^1_t$ as $O_t$.

\begin{lemma}[Lemma 1, \cite{MR2524684}]

For all $\omega\in\Omega$ and $t\in\mathbb{R}$, the integrals $e^{-\nu t}\int^t_{-\infty}e^{\nu s}dB^H_s(\omega)$ is well defined. Moreover, for all $\omega\in\Omega$ we have
\begin{equation}
|e^{-t}\int^t_{-\infty}{e^s}d{B^H_s(\omega)}|\leq K(\omega)(1+|t|)^2
\end{equation}
for all $\omega\in\Omega$.
\end{lemma}

It can easily be seen that $O_t$ has the sub-exponential growth from (2.9).

Moreover, by means of the integration by parts, it's also easy to show that $O_t, t\in\mathbb{R}$ has $\alpha$-H\"{o}lder continue regularity, where $\frac{1}{2}<\alpha<H$. In fact,
$$\begin{aligned}
    O_t-O_s
    &=\ e^{-t}\int^t_{-\infty}{e^r}d{B^H_r}-e^{-s}\int^s_{-\infty}{e^r}d{B^H_r}\\
    &=\ e^{-t}\int^t_s{e^r}d{B^H_r}+(e^{-t}-e^{-s})\int^s_{-\infty}{e^r}d{B^H_r}\\
    &=\ e^{-t}[(e^tB^H_t-e^sB^H_s)-\int^t_s{e^r}{B^H_r}dr]+(e^{-t}-e^{-s})[e^rB^H_r|^s_{-\infty}-\int^s_{-\infty}{e^r}d{B^H_r}],
\end{aligned}$$
note that we have
$$\lim_{r\rightarrow{-\infty}}e^rB^H_r=0$$
from (2.6), then the last term of equality above is well-defined. Furthermore, the H\"{o}lder continuity of $O_t$ is then obtained from $\alpha$-H\"{o}lder continuous regularity of fBm.

\begin{lemma}
    For all $\omega\in\Omega$, we have
    \begin{equation}
        \lim_{t\rightarrow\infty}\frac{1}{t}\int^t_0O_s(\omega)ds=0.
    \end{equation}
\end{lemma}
\begin{proof}
    We note that
    $$O_t(\omega)=e^{-t}\int^t_{-\infty}e^sd{B^H_s(\omega)}=B^H_t(\omega)-\int^t_{-\infty}e^{-(t-s)}B^H_s(\omega)ds,$$
    for simplicity, we denote by $\omega_t:=B^H_t(\omega)$, it is sufficient to prove
    \begin{equation}
        \lim_{t\rightarrow\infty}\frac{1}{t}\Big(\int^t_0\omega(s)ds-\int^t_0ds\int^s_{-\infty}e^{-(s-r)}\omega(r)dr\Big)=0.
    \end{equation}

    Note that
    \begin{align}
        \int^t_0ds\int^s_{-\infty}e^{-(s-r)}\omega(r)dr=-e^{-t}\int^t_{-\infty}e^r\omega(r)dr+\int^0_{-\infty}e^r\omega(r)dr+\int^t_0\omega(r)dr. \notag
    \end{align}

    Using (2.6), we get that
    \begin{align}
        \lim_{t\rightarrow\infty}\frac{1}{t}\int^0_{-\infty}e^r\omega(r)dr=0, \notag
    \end{align}
    it remains to show that
    \begin{equation}
        \lim_{t\rightarrow\infty}\frac{1}{t}\int^t_{-\infty}e^{-(t-s)}\omega(s)ds=0.
    \end{equation}

    Due to the sub-linear growth of fBm, for $\forall\epsilon>0$, $\exists $ $ T:=T(\epsilon,\omega)>0$, s.t. for $t>T$, $\omega(t)\leq\epsilon\cdot t$. Furthermore, $\omega(t)$ is a continuous function, hence we denote $C(\epsilon,\omega):=\sup_{t\in[-T,T]}|\omega(t)|$, then we have
    \begin{align}
        \frac{1}{t}\int^t_{-\infty}e^{-(t-s)}\omega(s)ds&=\frac{1}{t}\int^{-T}_{-\infty}e^{-(t-s)}\omega(s)ds+\frac{1}{t}\int^T_{-T}e^{-(t-s)}\omega(s)ds+\frac{1}{t}\int^t_Te^{-(t-s)}\omega(s)ds  \notag \\
        &\leq \frac{1}{t}\int^{-T}_{-\infty}e^{-(t-s)}\cdot\epsilon(-s)ds+\frac{1}{t}\int^T_{-T}e^{-(t-s)}\cdot C(\epsilon,\omega)ds+\frac{1}{t}\int^t_Te^{-(t-s)}\cdot\epsilon sds \notag \\
        &=\frac{\epsilon}{t}e^{-t}\int^{-T}_{-\infty}e^s|s|ds+\frac{1}{t}e^{-t}\int^T_{-T}e^s\cdot C(\epsilon,\omega)ds+\frac{\epsilon}{t}e^{-t}\int^t_Te^ssds  \notag \\
        &=: I_1+I_2+I_3,  \notag
    \end{align}
    where  $\int^{-T}_{-\infty}e^s|s|ds$ and   $\int^T_{-T}e^s\cdot C(\epsilon,\omega)ds$ are bounded, thus $\lim_{t\rightarrow\infty}I_1=\lim_{t\rightarrow\infty}I_2=0$.
    \begin{align}
        \frac{\epsilon}{t}e^{-t}\int^t_Te^ssds&=\frac{\epsilon}{t}e^{-t}\cdot(te^t-Te^T-e^t+e^T) \notag \\
        &\leq \epsilon\cdot(1+\frac{Te^T}{te^t}+\frac{1}{t}+\frac{e^T}{te^t}).  \notag
    \end{align}

    Since $\epsilon$ is arbitrarily chosen, $I_3$ converge to zero as $t\rightarrow\infty$.
\end{proof}

\begin{lemma}[Lemma 3.4, \cite{MR4549839}]
    For any $\delta>0$, we have
    \begin{align}
        \lim_{t\rightarrow\infty}|t|^{-\delta}|O_t|=0. \notag
    \end{align}
\end{lemma}

\section{Transformation between SDE and RDE}
In this section, we will use the fractional O-U process to constuct a transformation and it transforms the original equation into a new one which is easier to study. And then we will show that these two equations are equivalent.

Consider the stochastic differential equation (SDE for simplicity) in $\mathbb{R}^d$:
\begin{equation}
    dX_t=f(X_t)dt+(aX_t+b)dB^H_t,
\end{equation}
where $B^H_t$ is a two-sided one-dimensional fractional Brownian motion with Hurst paremeter $H>\frac{1}{2}$, $a\in\mathbb{R}, b\in\mathbb{R}^d$ and the function $f$ is a linear growth function which satisfying one-sided dissipative Lipschitz condition (1.3) and integrability condition (1.4).

We suppose $a\neq0$, define a map$$\mathcal{T}: X_t\mapsto e^{-aO_t}(X_t+\frac{b}{a})-\frac{b}{a},$$
where
$$dO_t=-O_tdt+dB^H_t,$$
$$O_t=e^{-t}\int^t_\infty{e^s}d{B^H_s}.$$

Denote
\begin{equation}
    U_t=e^{-aO_t}(X_t+\frac{b}{a})-\frac{b}{a}.
\end{equation}

Using this transformation we get random differential equation (RDE) from (3.1):
\begin{equation}
    \frac{dU}{dt}:=e^{-aO_t}f(U_te^{aO_t}+\frac{b}{a}(e^{aO_t}-1))+(a U_t+b)O_t.
\end{equation}

This transformation, although it seems reasonable, is in fact needed to verify that the derivation of functions with respect to O-U processes (or, more generally, stochastic integrals with respect to fBm) is capable of applying the chain rule. Before proceeding, we assume that $X_t$ and $O_t$ exhibit $\alpha$-H\"{o}lder continuity on the interval $[0,T]$, where $T>0$. We will provide a proof of this assumption at the end of this section. Now, we attempt to verify this $Young$ integral equation:
$$U_t-U_0=\int^t_0e^{-aO_r}dX_r-\int^t_0a(X_r+\frac{b}{a})e^{-aO_r}dO_r.$$
We informly rewrite it as differential form:
\begin{align}
        dU_t=\ e^{-aO_t}dX_t-a(X_t+\frac{b}{a})e^{-aO_t}dO_t.
\end{align}

For simplicity, let
    $$\widetilde{U}_t:=e^{-O_t}X_t.$$

To prove (3.4), it is sufficient to show that
\begin{equation}
    d\widetilde{U}_t=e^{-O_t}dX_t-e^{-O_t}X_tdO_t.
\end{equation}

In fact, take $t>0$,
\begin{align*}
    \widetilde{U}_t-\widetilde{U}_0
    &=e^{-O_t}X_t-e^{-O_0}X_0 \\
    &=\sum_{i=1}^{m}(e^{-O_{t_i}}X_{t_i}-e^{-O_{t_{i-1}}}X_{t_{i-1}}) \\
    &=\sum_{i=1}^{m}[e^{-O_{t_i}}(X_{t_i}-X_{t_{i-1}})+X_{t_{i-1}}(e^{-O_{t_i}}-e^{-O_{t_{i-1}}})] \\
    &=\sum_{i=1}^{m}e^{-O_{t_i}}(X_{t_i}-X_{t_{i-1}})+\sum_{i=1}^{m}X_{t_{i-1}}(e^{-O_{t_i}}-e^{-O_{t_{i-1}}}) \\
    &=: I_m+J_m,
\end{align*}
where $0=t_0<t_1<\cdots <t_{m-1}<t_m=t$ be any partition of the interval $[0,t]$.
\begin{align*}
    I_m
    &:=\sum_{i=1}^{m}e^{-O_{t_i}}(X_{t_i}-X_{t_{i-1}}) \\
    &=\sum_{i=1}^{m}e^{-O_{t_{i-1}}}(X_{t_i}-X_{t_{i-1}})+\sum_{i=1}^{m}(e^{-O_{t_i}}-e^{-O_{t_{i-1}}})(X_{t_i}-X_{t_{i-1}}).
\end{align*}

From the $\alpha$-H\"{o}lder continuity of $O_t$ and $X_t$, where $\frac{1}{2}<\alpha< H$, it is easy to obtain that
$$ \lim _{m\rightarrow\infty}\sum_{i=1}^{m}(e^{-O_{t_i}}-e^{-O_{t_{i-1}}})(X_{t_i}-X_{t_{i-1}})=0, $$
by the definition of $Young$ integral, $\lim _{m\rightarrow \infty}I_m=\int^t_0e^{-O_r}dX_r$.

On the other hand, using Taylor's expansion we can get that
\begin{align*}
 J_m
 &:=\sum_{i=1}^{m}X_{t_{i-1}}(e^{-O_{t_i}}-e^{-O_{t_{i-1}}}) \\
 &=\sum_{i=1}^{m}X_{t_{i-1}}[-e^{-O_{t_{i-1}}}(O_{t_i}-O_{t_{i-1}})+\frac{1}{2}e^{-O_{\tau}}(O_{t_i}-O_{t_{i-1}})^2], \ \ \tau\in(t_{i-1},t_i),
\end{align*}
where the second term of the series sum converges to 0 as $m\rightarrow\infty$, hence by the definition of $Young$ integral, $$\lim _{m\rightarrow \infty}J_m=\int^t_0X_r(-e^{-O_r})dO_r.$$

To summarize, we obtain that
$$\widetilde{U}_t-\widetilde{U}_0=\int^t_0e^{-O_r}dX_r+\int^t_0X_r(-e^{-O_r})dO_r,$$
and thus (3.5) is proved.

Use the similar method we can get (3.4). It follows that
\begin{align}
    U_t-U_0=&\int^t_0e^{-aO_s}dX_s-\int^t_0(aX_s+b)e^{-aO_s}dO_s \notag \\
    =&\int^t_0e^{-aO_s}f(X_s)ds+\int^t_0e^{-aO_s}(aX_s+b)dB^H_s \notag \\
    &-\int^t_0e^{-aO_s}(aX_s+b)(-O_s)ds-\int^t_0e^{-aO_s}(aX_s+b)dB^H_s \notag \\
    =&\int^t_0e^{-aO_s}[f(X_s)+(aX_s+b)O_s]ds \notag \\
    =&\int^t_0[e^{-aO_s}f(U_se^{aO_s}+\frac{b}{a}(e^{aO_s}-1))+(aU_s+b)O_s]ds.
\end{align}

According to the validity of the integral form (3.6), we naturally obtain the differential form (3.3).
\begin{remark}
    In the derivation above, it is worth noting that we have employed a property like:
    $$\int^t_0dX_s=\int^t_0f(X_s)ds+\int^t_0(aX_s+b)dB^H_s$$
    from (3.1). In fact, this can be readily verified through the definition of integration as below:
    \begin{align}
        \int^t_0dX_s&=\lim_{n\rightarrow\infty}\sum^n_{i=1}(X_{t_i}-X_{t_{i-1}})  \notag \\
        &=\lim_{n\rightarrow\infty}\sum^n_{i=1}\left( \int^{t_i}_{t_{i-1}}f(X_r)dr+\int^{t_i}_{t_{i-1}}(aX_r+b)dB^H_r \right) \notag \\
        &=\lim_{n\rightarrow\infty}\sum^n_{i=1}\int^{t_i}_{t_{i-1}}f(X_r)dr+\lim_{n\rightarrow\infty}\sum^n_{i=1}\left( \lim_{m\rightarrow\infty}\sum^m_{j=1}(aX_{t_{j-1}}+b)(B^H_{t_j}-B^H_{t_{j-1}}) \right)  \notag \\
        &=:\lim_{n\rightarrow\infty}I_n+\lim_{n\rightarrow\infty}\lim_{m\rightarrow\infty}J_{n,m} \notag
        \end{align}
        where $0=t_0<t_1<\cdots <t_n=t$ be any partition of the interval $[0,t]$ and $t_{i-1}=t_{j_0}<t_{j_1}<\cdots <t_{j_m}=t_i$ be any partition of the interval $[t_{i-1},t_i]$. Due to the additivity of integration,
        $$\lim_{n\rightarrow\infty}I_n=\int^t_0f(X_s)ds,$$
        and by the countability of natural numbers and the definition of $Young$ integral,
        $$\lim_{n\rightarrow\infty}\lim_{m\rightarrow\infty}J_{n,m}=\int^t_0(aX_s+b)dB^H_s.$$

\end{remark}
\vspace{11pt}
Next we discuss the equivalence of (3.1) and (3.3).
\begin{theorem}
    Let $U_t=e^{-aO_t}(X_t+\frac{b}{a})-\frac{b}{a}$. For any bounded time interval $[T_1,T_2]\subset\mathbb{R}$, the following statements are equivalent:
    \begin{enumerate}
    \item[(i)] $X_t\in C^{\alpha}([T_1,T_2],\mathbb{R}^d)$ satisfying (3.1).
    \item[(ii)]  $U_t\in C^{\alpha}([T_1,T_2],\mathbb{R}^d)$ satisfying (3.3).
    \end{enumerate}
\end{theorem}
\begin{proof}
First we assume that (i) holds, i.e., $X_t\in C^{\alpha}([T_1,T_2],\mathbb{R}^d)$, combining with the $\alpha$-H\"{o}lder continuity of $O_t$ and $X_t$, then
\begin{align*}
    \left \|U_t-U_s\right \|
    &=\left \|e^{-aO_t}(X_t+\frac{b}{a})-e^{-aO_s}(X_s+\frac{b}{a})\right \| \\
    &=\left \|e^{-aO_t}(X_t-X_s)+(X_s+\frac{b}{a})(e^{-aO_t}-e^{-aO_s})\right \| \\
    &\leq |e^{-aO_t}|\left \|X_t-X_s\right \|+\left \|X_s+\frac{b}{a}\right \|\sup_{s\leq r\leq t}|ae^{-aO_r}||O_t-O_s| \\
    &\leq C|t-s|^{\alpha},
\end{align*}
that is $U_t\in C^{\alpha}([T_1,T_2],\mathbb{R}^d)$.

Since $dX_t=f(X_t)dt+(aX_t+b)dB^H_t$, using (3.2) and (3.4) we can conclude that
\begin{align*}
    dU_t
    &=d(e^{-aO_t}(X_t+\frac{b}{a})-\frac{b}{a}) \\
    &=e^{-aO_t}dX_t-a(X_t+\frac{b}{a})e^{-aO_t}dO_t \\
    &=[e^{-aO_t}f(U_te^{aO_t}+\frac{b}{a}(e^{aO_t}-1))+(aU_t+b)O_t]dt,
\end{align*}
thus (ii) is proved.

In turn, we assume that (ii) holds. Similarly it can be deduced that $X_t\in C^{\alpha}([T_1,T_2],\mathbb{R}^d)$, and
\begin{align*}
    dX_t
    &=d(e^{aO_t}(U_t+\frac{b}{a})-\frac{b}{a}) \\
    &=e^{aO_t}dU_t+a(U_t+\frac{b}{a})e^{aO_t}dO_t \\
    &=f(X_t)dt+(aX_t+b)dB^H_t.
\end{align*}

That completes the proof of the theorem.
\end{proof}

Now we discuss the $\alpha$-H\"oler continuity of $X_t$. According to the equivalence stated in Theorem 3.1, we only need to prove the $\alpha$-H\"oler continuity of $U_t$. From (3.3), we have
\begin{align}
    U_t-U_s=\int^t_s[e^{-aO_r}f(U_re^{aO_r}+\frac{b}{a}(e^{aO_r}-1))+(a U_r+b)O_r]dr. \notag
\end{align}

By the linear growth condition of $f$ , it follows that
\begin{align}
    \left \|U_t-U_s\right \|&\leq\int^t_s[e^{-aO_r}\left\|f(U_re^{aO_r}+\frac{b}{a}(e^{aO_r}-1))\right\|+\left\|a U_r+b\right\| |O_r|]dr \notag \\
    &\leq \int^t_s[e^{-aO_r}(l\left\|U_re^{aO_r}+\frac{b}{a}(e^{aO_r}-1)\right\|+M)+\left\|a U_r+b\right\| |O_r|]dr \notag \\
    &\leq \int^t_s e^{-aO_r}|M+l\frac{\left \| b \right \|}{a}|dr+\int^t_s(|\frac{l}{a}|+|O_r|)\left \| aU_r+b \right \|dr.\notag
\end{align}

From Lemma 2.3, we have that $|O_r|\leq K(\omega)(1+\max\{|t|,|s|\})^2=:\Tilde{C}_{\omega,t,s}$ on any finite interval $[s,t]$, then
\begin{align}
    \left \| U_t-U_s \right \|\leq |M+l\frac{\left \| b \right \|}{a}|e^{-a\Tilde{C}_{\omega,t,s}}(t-s)+(|\frac{l}{a}|+\Tilde{C}_{\omega,t,s})\int^t_s\left \| aU_r+b \right \|dr.\notag
\end{align}

The second term on the right hand side is bounded on any finite time interval (see\cite{MR4549839}). For simplicity, we use a symbol $\hat{C}$ to cover all the constants appeared above, then we have $\left \| U_t-U_s \right \|\leq \hat{C}(t-s)$, which reveals the $\alpha$-H\"older continuity of $U_t$.

\section{Stability test of the uncoupled system}
In this section, we will prove the uncoupled SDEs (4.1) has a unique stochastic stationary solution.
\begin{align}
    dX_t=f(X_t)dt+(a_1X_t+b_1)dB^{H_2}_t,  \notag\\
    dY_t=g(Y_t)dt+(a_2Y_t+b_2)dB^{H_1}_t.
\end{align}

According to (3.3), the corresponding RDEs are:
\begin{align}
    \frac{dU}{dt}=e^{-a_1O^1_t}f(U_te^{a_1O^1_t}+\frac{b_1}{a_1}(e^{a_1O^1_t}-1))+(a_1U_t+b_1)O^1_t,  \notag \\
    \frac{dV}{dt}=e^{-a_2O^2_t}g(V_te^{a_2O^2_t}+\frac{b_2}{a_2}(e^{a_2O^2_t}-1))+(a_2V_t+b_2)O^2_t.
\end{align}

 We apply the method in \cite{MR3869887} and \cite{MR2399931} , but the noise employed here is a fractional Brownian motion. To ensure the existence and uniqueness of solutions, we suppose $f,g$ satisfy the sub-linear growth conditions and satisfy one-sided dissipative Lipschitz conditions
 \begin{align}
     \left \langle  x_1-x_2, f(x_1)-f(x_2) \right \rangle \leq -L\left \| x_1-x_2 \right \|^2, \notag \\
     \left \langle  y_1-y_2, g(y_1)-g(y_2) \right \rangle \leq -L\left \| y_1-y_2 \right \|^2,
 \end{align}
 on $\mathbb{R}^d$ for some $L>0$.

 From now on, we only consider the single equation.
    \begin{equation}
        dX_t=f(X_t)dt+(aX_t+b)dB^H_t.
    \end{equation}

    According to Theorem3.1, we next only need to prove the stationary of the solutions of RDE:
    \begin{equation}
        \frac{dU}{dt}=e^{-aO_t}f(U_te^{aO_t}+\frac{b}{a}(e^{aO_t}-1))+(aU_t+b)O_t,
    \end{equation}
    where
    $$O_t=e^{-t}\int^t_\infty{e^s}d{B^H_s}.$$

    The vector-field function
    $$\widetilde{f}(u,z)=e^{-az}f(ue^{az}+\frac{b}{a}(e^{az}-1))$$
    satisfies the one-sided Lipschitz condition with the first variable uniformly in the second variable. In fact,
    \begin{align*}
        \left \langle u_1-u_2, \widetilde{f}(u_1,z)-\widetilde{f}(u_2,z) \right \rangle &= \left \langle u_1-u_2, e^{-az}f(u_1e^{az}+\frac{b}{a}(e^{az}-1))-e^{-az}f(u_2e^{az}+\frac{b}{a}(e^{az}-1))\right \rangle \\
        &=e^{-2az}\left \langle e^{az}u_1-e^{az}u_2, f(u_1e^{az}+\frac{b}{a}(e^{az}-1))-f(u_2e^{az}+\frac{b}{a}(e^{az}-1))\right \rangle \\
        &\leq-L\left \| u_1-u_2 \right \|^2,
    \end{align*}
    the constant $L$  is the same as the Lipschitz constant for $f$.

    Thus we could know that any two solutions of RDE (4.5) satisfy
    \begin{align}
        \frac{d}{dt}\left \| U_1(t)-U_2(t) \right \|^2&=2\left \langle U_1(t)-U_2(t), \frac{d}{dt}U_1(t)-\frac{d}{dt}U_2(t)\right \rangle  \notag \\
        &\leq2(-L+aO_t)\left \| U_1(t)-U_2(t) \right \|^2.
    \end{align}

    Using the Gronwall's inequality, we have
    $$\left \| U_1(t)-U_2(t) \right \|^2\leq e^{-2t(L-\frac{1}{t}\int^t_0\alpha O_\tau d\tau)}\left \| U_1(0)-U_2(0) \right \|^2.$$

    Since fBm has polynomial growth, review to Lemma 2.4, that is
    \begin{equation}
        lim_{t\rightarrow\infty}\frac{1}{t}\int^t_0O(\theta_s\omega)ds=0,
    \end{equation}
    then we obtain
    $$lim_{t\rightarrow\infty}\left \| U_1(t)-U_2(t) \right \|^2=0,$$
    which means all solutions converge pathwise to each other.

    Note that the solution of (4.5) generates a random dynamical system $\psi=\psi(t,\omega,U_0):=U(t,\omega)$ over $(\Omega,\mathcal{F},\mathbb{P},(\theta_t)_{t\in\mathbb{R}})$ with state space $\mathbb{R}^d$, where $U_0$ is the initial value of the solution at time $t=0$. Then we need to show that the random attractors of $\psi$ are singleton sets. Omitting $\omega$ and $t$ for brevity, we have pathwise
    \begin{align}
        \frac{d}{dt}\left \| U \right \|^2&=2\left \langle U,\frac{d}{dt}U \right \rangle \notag\\
        &=2\left \langle U,e^{-aO_t}f(Ue^{aO_t}+\frac{b}{a}(e^{aO_t}-1))+(aU+b)O_t \right \rangle \notag\\
        &=2e^{-2aO_t}\left \langle Ue^{aO_t},f(Ue^{aO_t}+\frac{b}{a}(e^{aO_t}-1))-f(\frac{b}{a}(e^{aO_t}-1)) \right \rangle  \notag\\
        &+2e^{-aO_t}\left \langle U,f(\frac{b}{a}(e^{aO_t}-1)) \right \rangle+2a\left \| U \right \|^2O_t+2b\left \langle U,O_t \right \rangle \notag\\
        &\leq-2Le^{-2aO_t}\left \| Ue^{aO_t} \right \|^2+2aO_t\left \| U \right \|^2+2\left \langle U,e^{-aO_t}f(\frac{b}{a}(e^{aO_t}-1))+bO_t \right \rangle \notag\\
        &\leq (-L+2aO_t)\left \| U \right \|^2+\frac{1}{L}\left \| e^{-aO_t}f(\frac{b}{a}(e^{aO_t}-1))+bO_t \right \|^{2} \notag\\
        &\leq (-L+2aO_t)\left \| U \right \|^2+\frac{2}{L}(e^{-2aO_t}\left \| f(\frac{b}{a}(e^{aO_t}-1)) \right \|^{2}+\left \| bO_t \right \|^{2}).
    \end{align}

    By the Gronwall's inequality, one sees that
$$\left \| U_t \right \|^{2}\leq\left \| U_0 \right \|^{2}e^{-L(t-t_0)+2\int^t_{t_0}aO_\tau d\tau}+\frac{2}{L}\int^t_{t_0}(e^{-2a O_\tau}\left \| f(\frac{b}{a}(e^{aO_\tau}-1))  \right \|^{2}+\left \| b O_\tau \right \|^{2})e^{-L(t-\tau)+2\int^t_\tau aO_s ds}d\tau.$$

Use (4.7), we obtain that
$$e^{2\int^t_saO_\tau d\tau}\leq e^{\frac{L}{2}(t-s)},$$
when $s\leq 0\leq t$ and $\left | t \right |,\left | s \right |>T=T(\omega)$. Hence
$$\left \| U_t \right \|^{2}\leq\left \| U_0 \right \|^{2}e^{-\frac{L}{2}(t-t_0)}+\frac{2}{L}\int^t_{t_0}(e^{-2aO_\tau}\left \| f(\frac{b}{a}(e^{aO_\tau}-1))  \right \|^{2}+\left \| bO_\tau \right \|^{2})e^{-\frac{L}{2}(t-\tau)}d\tau$$
for $t_0\leq 0\leq t$ and $\left | t \right |,\left | t_0 \right |>T(\omega)$.

From Lemma 2.5 we can easily deduce the sub-exponential growth of $e^{aO_t}$. Combining with the integrability condition (1.4), we must have the integral pathwise
$$\frac{2}{L}\int^t_{-\infty}(e^{-2aO_\tau}\left \| f(\frac{b}{a}(e^{aO_\tau}-1))  \right \|^{2}+\left \| bO_\tau \right \|^{2})e^{-\frac{L}{2}(t-\tau)}d\tau<\infty.$$

Now we can use pullback convergence (i.e. with $t_0\rightarrow-\infty$) in the sense of pathwise to show that the closed ball centered at the origin with random radius
$$R^2(\omega):=1+\frac{2}{L}\int^0_{-\infty}(e^{-2a O_\tau(\omega)}\left \| f(\frac{b}{a}(e^{a O_\tau(\omega)}-1))  \right \|^{2}+\left \| b O_\tau(\omega) \right \|^{2})e^{L\tau+2\int^0_\tau a O_s(\omega) ds}d\tau$$
is a pullback absorbing set for $t>T(\omega)$. The theory of random dynamical system gives us a random attractor $\{\mathcal{A}(\omega),\omega\in\Omega\}$. As we proved above, all trajectories converge to each other forward in time, and thus the sets in this random attractor are singleton sets, i.e. $\mathcal{A}(\omega)=\{a(\omega)\}.$

When we transform (4.2) to (4.1), the SDE has the pathwise singleton set attractor $(a(\theta_t\omega)+\frac{b}{a}e^{O_t(\omega)}-\frac{b}{a})$, which is a stationary solution, since the fractional O-U process is stationary (see \cite{MR2524684}).

\section{Asymptotic behavior of the coupled system}
In this section, we will introduce our goal system and show that the coupled equations have the unique stochastic stationary solutions. These stationary solutions of coupled synchronized system are converge to each other while coupling intensity parameter $\kappa$ tend to infinity.

Consider the coupled RDEs:
\begin{equation}
\left\{\begin{array}{l}
\frac{dU}{dt}=F(U,O^1_t)+\kappa(V-U),\\
\frac{dV}{dt}=G(V,O^2_t)+\kappa(U-V),
\end{array}\right.
\end{equation}
with $\kappa>0$, and
$$F(U,O^1_t)=e^{-a_1O^1_t}f(Ue^{a_1 O^1_t}+\frac{b_1}{a_1}(e^{a_1 O^1_t}-1))+(a_1 U+b_1)O^1_t,$$
$$G(V,O^2_t)=e^{-a_2O^2_t}g(Ve^{a_2 O^2_t}+\frac{b_2}{a_2}(e^{a_2 O^2_t}-1))+(a_2 V+b_2)O^2_t,$$
where $f,g$ satisfy linear growth conditions (1.2), one-sided dissipative Lipschitz conditions (1.3) and integrability conditions (1.4). $O^1_t,O^2_t$ are fractional Ornstein-Uhlenbeck processes satisfying
$$dO^1_t=-O^1_tdt+dB^{H_1}_t,\ \ dO^2_t=-O^2_tdt+dB^{H_2}_t,$$
respectively, $B^{H_1}_t$ and $B^{H_2}_t$ are one-dimensional two-sided fractional Brownian motions with Hurst parameter $H_1,H_2\in(\frac{1}{2},1)$. $a_1,a_2\in\mathbb{R}$ and $b_1,b_2\in\mathbb{R}^d$. Since the general linear multiplicative noise is considered to be studied in this paper, we set $a_1,a_2\neq0$.

Similarly to (4.6), we can deduce that
\begin{align}
     &\frac{d}{dt}\left \| U_1(t)-U_2(t) \right \|^2\leq(-2L-\kappa+2a_1O^1_t)\left \| U_1(t)-U_2(t) \right \|^2+\kappa\left \| V_1(t)-V_2(t) \right \|^{2}, \notag\\
     &\frac{d}{dt}\left \| V_1(t)-V_2(t) \right \|^2\leq(-2L-\kappa+2a_2O^2_t)\left \| V_1(t)-V_2(t) \right \|^2+\kappa\left \| U_1(t)-U_2(t) \right \|^{2}.
\end{align}

Define the matrix
$$A_\kappa(t)=\begin{pmatrix}
-2L-\kappa+2a_1O^1_t & \kappa\\
\kappa & -2L-\kappa+2a_2O^2_t
\end{pmatrix}, \ \ t\in\mathbb{R},
$$
and vector
$$m(t)=\begin{pmatrix}
\left \| U_1(t)-U_2(t) \right \|^2\\
\left \| V_1(t)-V_2(t) \right \|^2
\end{pmatrix}, \ \ t\in\mathbb{R},
$$
then (5.2) can rewrite as
\begin{equation}
    \dot{m}(t)\leq A_\kappa(t)m(t),\ \ m(t)\in\mathbb{R}^2. \notag
\end{equation}

Since the off-diagonal entries of matrix $A_\kappa(t)$ are positive and independent of $t\in\mathbb{R}$, see Corollary 1.5.2 in \cite{MR0379933}, we obtain that
\begin{equation}
    m(t)\leq e^{\int^t_0A_\kappa(\tau)d\tau}m(0)
\end{equation}
componentwise.

Now we need the following lemma, see \cite{MR2399931}.
\begin{lemma}
    As $A_\kappa(t)$ is described above, there exists $T(\omega)>0$ such that for $t\geq T(\omega)$, and for all $\kappa>1$ we have
    $$\left \| e^{\int^t_0A_\kappa(\tau)d\tau}m \right \|\leq e^{-Lt}\left \| m \right \|, \ \ m\in\mathbb{R}^2.$$
\end{lemma}
\begin{proof}
    Firstly we note that the matrix $\int^t_0A_\kappa(\tau)d\tau$ is symmetric, so there exists an orthonormal basis of eigenvectors $u^{(1)}_{\kappa,t},u^{(2)}_{\kappa,t}$ with eigenvalues $\lambda^{(1)}_{\kappa,t},\lambda^{(2)}_{\kappa,t}$ such that
    $$e^{\int^t_0A_\kappa(\tau)d\tau}m=e^{\lambda^{(1)}_{\kappa,t}}\cdot u^{(1)}_{\kappa,t}\cdot c^{(1)}_{m,\kappa,t}+e^{\lambda^{(2)}_{\kappa,t}}\cdot u^{(2)}_{\kappa,t}\cdot c^{(2)}_{m,\kappa,t},$$
    where
    $$u^{(1)}_{\kappa,t}\cdot c^{(1)}_{m,\kappa,t}+u^{(2)}_{\kappa,t}\cdot c^{(2)}_{m,\kappa,t}=m.$$

    Consequently we obtain
    \begin{align}
        \left \| e^{\int^t_0A_\kappa(\tau)d\tau}m \right \|^2&=e^{2\lambda^{(1)}_{\kappa,t}}\left \| u^{(1)}_{\kappa,t}\cdot c^{(1)}_{m,\kappa,t} \right \|^{2}+e^{2\lambda^{(2)}_{\kappa,t}}\left \| u^{(2)}_{\kappa,t}\cdot c^{(2)}_{m,\kappa,t} \right \|^{2} \notag \\
        &\leq e^{2\max\{\lambda^{(1)}_{\kappa,t},\lambda^{(2)}_{\kappa,t}\}}\left \| m \right \|^{2}.
    \end{align}

    Note that the eigenvalues of $\int^t_0A_\kappa(\tau)d\tau$ are given by
    $$\lambda^{(1),(2)}_{\kappa,t}=(-2L-\kappa)t+\int^t_0(\alpha_1O^1_t+\alpha_2O^2_t)d\tau\pm\sqrt{(\int^t_0(\alpha_1O^1_t-\alpha_2O^2_t)^2d\tau)^2+\kappa^2t^2}.$$

    Use the property (2.10) of fractional O-U process, it follows that
    \begin{equation}
        \lambda^{(1),(2)}_{\kappa,t}\leq-Lt,
    \end{equation}
    for $\left | t \right |\geq T(\omega)$ and all $\kappa\geq1.$

    Combing (5.4) and (5.5) yields the assertion.
\end{proof}

The above lemma implies that
\begin{equation}
    \lim_{t\rightarrow\infty}\left \| U_1(t)-U_2(t) \right \|^2=\lim_{t\rightarrow\infty}\left \| V_1(t)-V_2(t) \right \|^2=0,
\end{equation}
thus all solutions of the coupled RDEs (5.1) converge pathwise to each other forward in time.

Now we can use the theory of random dynamical systems to show what the solutions of (5.1) will converge to. The solution $\psi(t,\omega,\hat{m}_0):=\hat{m}(t,\omega)=\begin{pmatrix}
U(t,\omega)\\
V(t,\omega)
\end{pmatrix}$ with initial value $\hat{m}_0=\begin{pmatrix}
U(t_0)\\
V(t_0)
\end{pmatrix}$ at time $t=0$ of (5.1) generates a random dynamical ststem over $(\Omega,\mathcal{F},\mathbb{P},(\theta_t)_{t\in\mathbb{R}})$ with state space $\mathbb{R}^{2d}$.

Using the same argument as in the discussion (4.8), omitting $t$ and $\omega$ for brevity, we have
\begin{align}
    \frac{d}{dt}\left \| U \right \|^2&=2\left \langle U,\frac{d}{dt}U \right \rangle \notag\\
        &=2\left \langle U,e^{-a_1 O^1_t}f(Ue^{a_1 O^1_t}+\frac{b_1}{a_1}(e^{a_1 O^1_t}-1))+(a_1 U+b_1)O^1_t+\kappa(V-U) \right \rangle \notag\\
        &\leq(-L-\kappa+2a_1O^1_t)\cdot\left \| U \right \|^{2}+\kappa\left \| V \right \|^{2} \notag \\
        &+\frac{2}{L}(e^{-2a_1O^1_t}\cdot\left \| f(\frac{b_1}{a_1}(e^{a_1 O^1_t}-1)) \right \|^{2}+\left \| b_1O^1_t \right \|^{2}), \notag \\
    \frac{d}{dt}\left \| V \right \|^2&\leq(-L-\kappa+2a_2O^2_t)\cdot\left \| V \right \|^{2}+\kappa\left \| U \right \|^{2} \notag \\
    &+\frac{2}{L}(e^{-2a_2O^2_t}\cdot\left \| g(\frac{b_2}{a_2}(e^{a_2 O^2_t}-1)) \right \|^{2}+\left \| b_2O^2_t \right \|^{2}),  \notag
\end{align}
and then take $$n(t):=\begin{pmatrix}
\left \| U(t) \right \|^{2}\\
\left \| V(t) \right \|^{2}
\end{pmatrix},\ \  \hat{A}_\kappa(t)=\begin{pmatrix}
-L-\kappa+2a_1O^1_t & \kappa\\
\kappa & -L-\kappa+2a_2O^2_t
\end{pmatrix},$$ \\
we have componentwise
\begin{equation}
    n(t)\leq e^{\int^t_{t_0}\hat{A}_\kappa(\tau)d\tau}\cdot n(t_0)+\frac{2}{L}\int^t_{t_0}e^{\int^t_{\tau}\hat{A}_\kappa(s)ds}\begin{pmatrix}
e^{-2a_1O^1_\tau}\cdot\left \| f(\frac{b_1}{a_1}(e^{a_1 O^1_\tau}-1)) \right \|^{2}+\left \| b_1O^1_\tau \right \|^{2}\\
e^{-2a_2O^2_\tau}\cdot\left \| g(\frac{b_2}{a_2}(e^{a_2 O^2_\tau}-1)) \right \|^{2}+\left \| b_2O^2_\tau \right \|^{2}
\end{pmatrix}d\tau.
\end{equation}

By analogy with Lemma 5.1, we have the following lemma:
\begin{lemma}
    Let $t_0\leq0$, $t\geq0$. Then
    $$\left \| e^{\int^t_{t_0}\hat{A}_\kappa(\tau)d\tau}\cdot n \right \|\leq e^{-\frac{L}{2}(t-t_0)}\cdot\left \| n \right \|, \ \ n\in\mathbb{R}^2$$
    for $\left | t_0 \right |, \left | t \right |\geq T(\omega)$ and all $\kappa\geq1.$
\end{lemma}

Now we set
$$C_\kappa(\omega):=\frac{2}{L}\int^0_{-\infty}e^{\int^0_{\tau}\hat{A}_\kappa(s)ds}\begin{pmatrix}
e^{-2a_1O^1_\tau}\cdot\left \| f(\frac{b_1}{a_1}(e^{a_1 O^1_\tau}-1)) \right \|^{2}+\left \| b_1O^1_\tau \right \|^{2}\\
e^{-2a_2O^2_\tau}\cdot\left \| g(\frac{b_2}{a_2}(e^{a_2 O^2_\tau}-1)) \right \|^{2}+\left \| b_2O^2_\tau \right \|^{2}
\end{pmatrix}d\tau$$
and define
$$R^2_\kappa(\omega):=1+\left \| C_\kappa(\omega) \right \|^{2}.$$

According to (5.7) and pullback techniques, we see that the random ball $B_\kappa(\omega)$ in $\mathbb{R}^{2d}$ centered on the origin and with radius $R_\kappa(\omega)$ are pullback absorbing.

In addition to this, we note that
$$R_\kappa(\omega)\leq R_1(\omega),\ \ for\ \kappa\geq1,$$
indeed
$$\frac{d}{d\kappa}\left \| C_\kappa(\omega) \right \|^2=2\left \langle C_\kappa(\omega),\frac{d}{d\kappa}C_\kappa(\omega) \right \rangle=2\left \langle C_\kappa(\omega),\begin{pmatrix}
-1 & 1\\
1 & -1
\end{pmatrix}\right \rangle\leq0.$$

Hence the random dynamical system generated by RDEs (5.1) for each $\kappa\geq1$ has a random attractor $\{\mathcal{A}_\kappa(\omega),\omega\in\Omega\}$ in $B_\kappa(\omega)$, more precisely, $\mathcal{A}_\kappa(\omega)\subset B_1(\omega)$ for $\kappa\geq1$. Note that all solutions converge to each other pathwise forward in time. Thus $\mathcal{A}_\kappa(\omega)$ are singleton sets, say $\mathcal{A}_\kappa(\omega)=\{(\Bar{U}_\kappa(\omega),\Bar{V}_\kappa(\omega))\}$.

In order to the next discussion, we now proceed to estimate the difference between the two components of the coupled system(5.1). It can easily be shown that pathwise
\begin{align}
    \frac{d}{dt}\left \| U-V \right \|^2&=2\left \langle U-V,\frac{d}{dt}U-\frac{d}{dt}V \right \rangle \notag \\
    &=2\left \langle U-V,e^{-a_1 O^1_t}f(Ue^{a_1 O^1_t}+\frac{b_1}{a_1}(e^{a_1 O^1_t}-1))-e^{-a_2 O^2_t}g(Ve^{a_2 O^2_t}+\frac{b_2}{a_2}(e^{a_2 O^2_t}-1))\right \rangle \notag \\
    &+2\left \langle U-V,a_1UO^1_t-a_2VO^2_t\right \rangle +2\left \langle U-V,b_1O^1_t-b_2O^2_t\right \rangle+2\left \langle U-V,2\kappa(V-U)\right \rangle \notag \\
    &\leq-4\kappa\left \| U-V \right \|^{2}+2\left \| U-V \right \|\cdot\Big ( e^{-a_1 O^1_t} \left \| f(Ue^{a_1 O^1_t}+\frac{b_1}{a_1}(e^{a_1 O^1_t}-1))\right \| \notag \\
    &+e^{-a_2 O^2_t}\left \| g(Ve^{a_2 O^2_t}+\frac{b_2}{a_2}(e^{a_2 O^2_t}-1)) \right \|+\left \| a_1UO^1_t-a_2VO^2_t \right \|^{2}+\left \| b_1O^1_t-b_2O^2_t \right \|^{2} \Big ) \notag \\
    &\leq-\kappa\left \| U-V \right \|^{2}+\frac{1}{\kappa}e^{-2a_1 O^1_t}\left \| f(Ue^{a_1 O^1_t}+\frac{b_1}{a_1}(e^{a_1 O^1_t}-1))\right \|^2 \notag \\
    &+\frac{1}{\kappa}e^{-2a_2 O^2_t}\left \| g(Ve^{a_2 O^2_t}+\frac{b_2}{a_2}(e^{a_2 O^2_t}-1))\right \|^2 \notag \\
    &+\frac{1}{\kappa}\left \| b_1O^1_t-b_2O^2_t \right \|^{2}+\frac{1}{\kappa}\left | a_1O^1_t \right |^2\left \| U \right \|^{2}+\frac{1}{\kappa}\left | a_2O^2_t \right |^2\left \| V \right \|^{2}.  \notag
\end{align}

Denote
\begin{align}
    M^\kappa_{T_1,T_2,\omega}:=&\sup_{t\in[T_1,T_2]}\left( e^{-2a_1 O^1_t}\left \| f(Ue^{a_1 O^1_t}+\frac{b_1}{a_1}(e^{a_1 O^1_t}-1))\right \|^2+\left | a_1O^1_t \right |^2\left \| U \right \|^{2} \right) \notag \\
    &+\sup_{t\in[T_1,T_2]}\left( e^{-2a_2 O^2_t}\left \| g(Ve^{a_2 O^2_t}+\frac{b_2}{a_2}(e^{a_2 O^2_t}-1))\right \|^2+\left | a_2O^2_t \right |^2\left \| V \right \|^{2} \right) \notag \\
    &+\sup_{t\in[T_1,T_2]}\left( \left \| b_1O^1_t-b_2O^2_t \right \|^{2} \right),  \notag
\end{align}
we have
$$\frac{d}{dt}\left \| U_\kappa-V_\kappa \right \|^2\leq-\kappa\left \| U_\kappa-V_\kappa \right \|^2+\frac{1}{\kappa}M^\kappa_{T_1,T_2,\omega},$$
where the subscript $\kappa$ to indicate this dependence.

Without loss of generality, we can restrict solutions in the compact absorbing balls $B_\kappa(\omega)$ and thereby in the common compact absorbing ball $B_1(\omega)$ for $\kappa\geq1$. Hence $M^\kappa_{T_1,T_2,\omega}$ is uniformly bounded in $\kappa$, and
$$\frac{d}{dt}\left \| U_\kappa-V_\kappa \right \|^2\leq-\kappa\left \| U_\kappa-V_\kappa \right \|^2+\frac{1}{\kappa}M_{T_1,T_2,\omega},$$
where
$$M_{T_1,T_2,\omega}=\sup_{\kappa\leq1}M^\kappa_{T_1,T_2,\omega}.$$

As a result, we conclude that
$$\lim_{\kappa\rightarrow\infty}\left \| U_\kappa(t)-V_\kappa(t) \right \|^2=0,$$
for all $t\in [T_1,T_2]\subseteq \mathbb{R}$.

\section{Synchronization}
The aim of this section is to derive results on the synchronization of the coupled SDE systems as the coupling intensity parameter $\kappa$ tends to infinity. Before approaching this problem, we first pose the problem of finding what the coupled RDE systems (5.1) synchronize with.

\begin{lemma}
    Let $f,g$ satisfy linear growth conditions (1.2), one-sided dissipative Lipschitz conditions (1.3) and integrability conditions (1.4). $O^1_t,O^2_t$ are fractional O-U processes satisfy
    $$dO^1_t=-O^1_tdt+dB^{H_1}_t,\ \ dO^2_t=-O^2_tdt+dB^{H_2}_t,$$
    then the averaged equation
    \begin{align}
            \frac{dW}{dt}=&\frac{1}{2}[e^{-a_1O^1_t}f(e^{a_1O^1_t}W+\frac{b_1}{a_1}(e^{a_1O^1_t}-1))+e^{-a_2O^2_t}g(e^{a_2O^2_t}W+\frac{b_2}{a_2}(e^{a_2O^2_t}-1))] \notag \\
            &+\frac{1}{2}[(a_1W+b_1)O^1_t+(a_2W+b_2)O^2_t]
    \end{align}
has a unique stationary solution $\Bar{W}_t\in C^{\alpha}([T_1,T_2],\mathbb{R}^d)$ for any bounded interval $[T_1,T_2]\subset\mathbb{R}$, where $\alpha< \min\{H_1,H_2\}$.
\end{lemma}

The proof is similar to Section 4, and is omitted here. It can easily be seen that the random dynamical system $\phi$ generated by $\Bar{W}_t(\omega)$ has a random singleton attractor. Next, we present the synchronization results of coupled system (5.1). Before proceeding, we present the following lemma:
\begin{lemma}
    Let $\{x_n\}$ be a sequence in a complete metric space $(X,d)$ such that every subsequence $\{x_{n_i}\}$ has a subsequence $\{x_{n_{i_j}}\}$ converging to a common limit $x^*$. Then the sequence $\{x_n\}$ converges to $x^*$.
\end{lemma}
\begin{theorem}
    Let ${\kappa_n}$ be a sequence in $\mathbb{R}$, then $\left( \Bar{U}_{\kappa_n}(t,\omega), \Bar{V}_{\kappa_n}(t,\omega) \right) \rightarrow \left( \Bar{W}(t,\omega), \Bar{W}(t,\omega) \right)$ pathwise uniformly on bounded time intervals $[T_1,T_2]$ on $\mathbb{R}$ for any sequence $\kappa_n\rightarrow\infty$.
\end{theorem}
\begin{proof}
    Define $$\Bar{W}_\kappa(t,\omega):=\frac{1}{2}\left( \Bar{U}_\kappa(t,\omega)+\Bar{V}_\kappa(t,\omega) \right).$$
    It is easy to see that
    \begin{align}
        \frac{d\Bar{W}_\kappa}{dt}=&\frac{1}{2}[e^{-a_1O^1_t}f(e^{a_1O^1_t}\Bar{U}_\kappa+\frac{b_1}{a_1}(e^{a_1O^1_t}-1))+e^{-a_2O^2_t}g(e^{a_2O^2_t}\Bar{V}_\kappa+\frac{b_2}{a_2}(e^{a_2O^2_t}-1))] \notag \\
            &+\frac{1}{2}[(a_1\Bar{U}_\kappa+b_1)O^1_t+(a_2\Bar{V}_\kappa+b_2)O^2_t]. \notag
    \end{align}

    Thus
    $$\sup_{t\in[T_1,T_2]} \frac{d}{dt}\left \| \Bar{W}_\kappa(t,\omega) \right \|^2 \leq M_{T_1,T_2,\omega}<\infty$$
    by continuity and the fact that these solutions belong to the common compact ball $B_1(\omega)$. We can use the Arzelà–Ascoli theorem to deduce that there is a subsequence $\kappa_{n_j}\rightarrow\infty$ such that $\left \|\Bar{W}_{\kappa_{n_j}}(t,\omega)\right \|\rightarrow\left \|\Bar{W}(t,\omega)\right \|$ as $\kappa_{n_j}\rightarrow\infty$. Meanwhile we have
    $$\left \|\Bar{W}_{\kappa_{n_j}}(t,\omega)-\Bar{V}_{\kappa_{n_j}}(t,\omega)\right \|=\frac{1}{2}\left \| \Bar{U}_{\kappa_{n_j}}(t,\omega)-\Bar{V}_{\kappa_{n_j}}(t,\omega) \right \|\rightarrow0,$$
    $$\left \|\Bar{W}_{\kappa_{n_j}}(t,\omega)-\Bar{U}_{\kappa_{n_j}}(t,\omega)\right \|=\frac{1}{2}\left \| \Bar{V}_{\kappa_{n_j}}(t,\omega)-\Bar{U}_{\kappa_{n_j}}(t,\omega) \right \| \rightarrow0,$$
    as $\kappa_{n_j}\rightarrow\infty$. This is attributed to the demonstration provided in the previous section.

    So
    $$\Bar{U}_{\kappa_{n_j}}(t,\omega)=2\Bar{W}_{\kappa_{n_j}}(t,\omega)-\Bar{V}_{\kappa_{n_j}}(t,\omega)\rightarrow\Bar{W}(t,\omega),$$
    $$\Bar{V}_{\kappa_{n_j}}(t,\omega)=2\Bar{W}_{\kappa_{n_j}}(t,\omega)-\Bar{U}_{\kappa_{n_j}}(t,\omega)\rightarrow\Bar{W}(t,\omega),$$
    as $\kappa_{n_j}\rightarrow\infty$.

    Note that $\Bar{W}_\kappa(t,\omega)$ satisfies the following integral expression:
    \begin{align}
        \Bar{W}_\kappa(t,\omega)=&\Bar{W}_\kappa(T_1,\omega)+\frac{1}{2}\int^t_{T_1}e^{-a_1O^1_s(\omega)}f(e^{a_1O^1_s(\omega)}\Bar{U}_\kappa(s,\omega)+\frac{b_1}{a_1}(e^{a_1O^1_s(\omega)}-1))ds \notag \\
        &+\frac{1}{2}\int^t_{T_1}e^{-a_2O^2_s(\omega)}g(e^{a_2O^2_s(\omega)}\Bar{V}_\kappa(s,\omega)+\frac{b_2}{a_2}(e^{a_2O^2_s(\omega)}-1))ds \notag \\
        &+\frac{1}{2}\int^t_{T_1}(a_1\Bar{U}_\kappa(s,\omega)+b_1)O^1_s(\omega)+(a_2\Bar{V}_\kappa(s,\omega)+b_2)O^2_t(\omega)ds.
    \end{align}

    It follows that the $\kappa_{n_j}$ subsequence converges pathwise to
        \begin{align}
        \Bar{W}(t,\omega)=&\Bar{W}(T_1,\omega)+\frac{1}{2}\int^t_{T_1}e^{-a_1O^1_s(\omega)}f(e^{a_1O^1_s(\omega)}\Bar{W}(s,\omega)+\frac{b_1}{a_1}(e^{a_1O^1_s(\omega)}-1))ds \notag \\
        &+\frac{1}{2}\int^t_{T_1}e^{-a_2O^2_s(\omega)}g(e^{a_2O^2_s(\omega)}\Bar{W}(s,\omega)+\frac{b_2}{a_2}(e^{a_2O^2_s(\omega)}-1))ds \notag \\
        &+\frac{1}{2}\int^t_{T_1}(a_1\Bar{W}(s,\omega)+b_1)O^1_s(\omega)+(a_2\Bar{W}(s,\omega)+b_2)O^2_t(\omega)ds,
    \end{align}
    on  any bounded time interval $[T_1,T_2]$, so $\Bar{W}(t,\omega)$ is a solution of RDE (6.1) for all $t\in\mathbb{R}$. Moreover, we note that all possilbe subsequence here have the same limit. Thus, by Lemma 6.2, every full sequence $\Bar{W}_\kappa(t,\omega)$ actually converges to $\Bar{W}(t,\omega)$ for the whole sequence $\kappa_n\rightarrow\infty$.
\end{proof}

As a straight consequence of the arguments in the previous proof we have
\begin{corollary}
    $\left( \Bar{U}_\kappa(t,\omega), \Bar{V}_\kappa(t,\omega) \right) \rightarrow \left( \Bar{W}(t,\omega), \Bar{W}(t,\omega) \right)$ pathwise uniformly on bounded time intervals $[T_1,T_2]$ on $\mathbb{R}$ as $\kappa\rightarrow\infty$.
\end{corollary}

Recall the transformation introduced in Section 3, let
\begin{align}
    U_t=e^{-a_1O^1_t}(X_t+\frac{b_1}{a_1})-\frac{b_1}{a_1},  \notag\\
    V_t=e^{-a_2O^2_t}(Y_t+\frac{b_2}{a_2})-\frac{b_2}{a_2},
\end{align}

For RDEs (5.1), using (6.4) and (3.4), then we obtain the coupled SDEs below:
\begin{equation}
\left\{\begin{array}{l}
dX_t=[f(X_t)+\kappa(e^{2\eta_t}Y_t-X_t)+\kappa(\frac{b_2}{a_2}(e^{2\eta_t}-e^{a_1O^1_t})-\frac{b_1}{a_1}(1-e^{a_1O^1_t}))]dt+(a_1X_t+b_1)dB^{H_1}_t,\\
 \\
dY_t=[g(Y_t)+\kappa(e^{-2\eta_t}X_t-Y_t)+\kappa(\frac{b_1}{a_1}(e^{-2\eta_t}-e^{a_2O^2_t})-\frac{b_2}{a_2}(1-e^{a_2O^2_t}))]dt+(a_2Y_t+b_2)dB^{H_2}_t,
\end{array}\right.
\end{equation}
where $\eta_t=\frac{1}{2}(a_1O^1_t-a_2O^2_t)$.

According to the previous results, it follows that coupled systems (6.5) has a unique stochastic stationary solution $\left(\Bar{X}_\kappa(t),\Bar{Y}_\kappa(t)\right)$, which is pathwise globally asymptotically stable with
\begin{align}
\begin{pmatrix}
\Bar{X}_\kappa(t,\omega)\\
\Bar{Y}_\kappa(t,\omega)
\end{pmatrix}\rightarrow
\begin{pmatrix}
e^{\eta_t}\Bar{Z}(t,\omega)+\frac{b_1}{a_1}(e^{a_1O^1_t(\omega)}-1)\\
e^{-\eta_t}\Bar{Z}(t,\omega)+\frac{b_2}{a_2}(e^{-a_2O^2_t(\omega)}-1)
\end{pmatrix}
,\ \  \kappa\rightarrow\infty,
\end{align}
where $\Bar{Z}(t,\omega)$ is the pathwise asymptotically stable solution of the following \emph{averaged} SDE:
\begin{align}
    dZ_t=&\frac{1}{2}[e^{-\eta_t}f\Big(^{\eta_t}Z_t+\frac{b_1}{a_1}(e^{a_1O^1_t}-1)\Big)+e^{\eta_t}g\Big(^{-\eta_t}Z_t+\frac{b_2}{a_2}(e^{a_2O^2_t}-1)\Big)]dt
 \notag\\
    &+\frac{1}{2}e^{\frac{1}{2}(a_1O^1_t+a_2O^2_t)}((b_1O^1_t+b_2O^2_t))dt  \notag\\
    &+\frac{1}{2}a_1Z_tdB^{H_1}_t+\frac{1}{2}a_2Z_tdB^{H_2}_t.
\end{align}

\section{Two examples}
In this section, we discuss two special cases respectively and illustrate the corresponding results of synchronization.
\subsection{Case 1: Pure multiplicative noise}
Consider the case that $b_1=b_2=0$. Now the coupled system (6.5) reduces to the coupled SDEs:
\begin{equation}
\left\{\begin{array}{l}
dX_t=[f(X_t)+\kappa(e^{2\eta_t}Y_t-X_t)]dt+a_1X_tdB^{H_1}_t,\\
 \\
dY_t=[g(Y_t)+\kappa(e^{-2\eta_t}X_t-Y_t)]dt+a_2Y_tdB^{H_2}_t,
\end{array}\right.
\end{equation}
where $a_1,a_2\in\mathbb{R}$, $a_1,a_2\neq0$, $\eta_t=\frac{1}{2}(a_1O^1_t-a_2O^2_t)$, and
\begin{equation}
    O^1_t=e^{-t}\int^t_{-\infty}{e^s}d{B^{H_1}_s},\quad O^2_t=e^{-t}\int^t_{-\infty}{e^s}d{B^{H_2}_s}. \notag
\end{equation}

The functions $f,g$ satisfy linear growth conditions (1.2), one-sided dissipative Lipschitz conditions (1.3) and integrability conditions (1.4). By employing the method of variables substitution in conjunction with the corresponding RDEs system, we obtain the following results.

The system (7.1) has a unique stationary solution $(\Bar{X}_\kappa,\Bar{Y}_\kappa)$, which is pathwise globally asymptotically stable, i.e.
\begin{center}
    $(\Bar{X}_\kappa(\theta_t\omega),\Bar{Y}_\kappa(\theta_t\omega))\rightarrow(\Bar{Z}(\theta_t\omega)\cdot e^{-a_1O^1_t(\omega)}, \Bar{Z}(\theta_t\omega)\cdot e^{-a_2O^2_t(\omega)})$, as $\kappa\rightarrow\infty,$
\end{center}
where the convergence on any interval $[T_1,T_2]\subseteq \mathbb{R}$ is in the sense of pathwise.

$\Bar{Z}_t(\omega)$ is the pathwise asymptotically stable solution of the averaged RDE:
\begin{equation}
    \frac{dZ}{dt}=\frac{1}{2}\Big[ e^{-a_1O^1_t(\omega)}f(e^{a_1O^1_t(\omega)}Z)+e^{-a_2O^2_t(\omega)}g(e^{a_2O^2_t(\omega)}Z)+(a_1O^1_t(\omega)+a_2O^2_t(\omega))Z \Big].
\end{equation}
\\

\subsection{Case 2: Mixed noise}
We consider the stochastic differential equations with mixed noise in $\mathbb{R}^d$:
\begin{align}
dX_t&=f(X_t)dt+(aX_t+b_1)dB^{H_1}_t, \notag \\
dY_t&=g(Y_t)dt+b_2dB^{H_2}_t,
\end{align}
where $B^{H_1}_t$ and $B^{H_2}_t$ are one-dimensional two-sided fractional Brownian motions with Hurst parameter $H_1,H_2\in(\frac{1}{2},1)$, $0\neq a\in\mathbb{R}$, $b_1,b_2\in\mathbb{R}^d$. The functions $f,g$ satisfy the same conditions as metioned above.

Below we discuss the synchronization results using the same method as described in the previous sections. Considering the transformation:
\begin{align}
    U_t(\omega)=e^{-aO^1_t(\omega)}\cdot(X_t(\omega)+\frac{b_1}{a})-\frac{b_1}{a}, \quad     V_t(\omega)=Y_t(\omega)-b_2O^2_t(\omega),
\end{align}
where
$$O^1_t=e^{-t}\int^t_{-\infty}{e^s}d{B^{H_1}_s},\quad O^2_t=e^{-t}\int^t_{-\infty}{e^s}d{B^{H_2}_s},\quad t\in\mathbb{R},$$
are two fractional Orstein-Uhlenck processes. Then (7.3) transforms to the pathwise random differential equations:
\begin{align}
    \frac{dU}{dt}&=F(U,O^1_t):=e^{-aO^1_t}f(e^{aO^1_t}(U+\frac{b_1}{a})+(aU+b_1)O^1_t,  \notag \\
    \frac{dV}{dt}&=G(V,O^2_t):=g(V+b_2O^2_t)+b_2O^2_t.
\end{align}

By linear cross coupling, we consider the coupled RDEs:
\begin{equation}
\left\{\begin{array}{l}
 \frac{dU}{dt}=F(U,O^1_t)+\kappa(V-U),\\
 \frac{dV}{dt}=G(V,O^2_t)+\kappa(U-V).
\end{array}\right.
\end{equation}

System (7.6) has a random attractor consist of a single stochastic stationary process $\Big(\Bar{U}_\kappa(\omega),\Bar{V}_\kappa(\omega)\Big)$. In particular,  $\Big(\Bar{U}_\kappa(\omega),\Bar{V}_\kappa(\omega)\Big)\rightarrow\Big( \Bar{W}(\omega),\Bar{W}(\omega)\Big)$, as $\kappa\rightarrow\infty$, where $\Bar{W}(\omega)$ is the pathwise asymptotically stable solution of the RDE:
\begin{equation}
    \frac{dW}{dt}=\frac{1}{2}\Big[e^{-aO^1_t}f(e^{aO^1_t}(W+\frac{b_1}{a}))+g(W+b_2O^2_t)+(aW+b_1)O^1_t+b_2O^2_t\Big].
\end{equation}

In terms of the original system (7.3), we have the coupled SDEs:
\begin{equation}
\left\{\begin{array}{l}
dX_t=\Big[f(X_t)+\kappa(e^{aO^1_t}Y_t-X_t)+\kappa(\frac{b_1}{a}(e^{aO^1_t}-1)-e^{aO^1_t}b_2O^2_t)\Big]dt+(aX_t+b_1)dB^{H_1}_t,\\
 \\
dY_t=\Big[g(Y_t)+\kappa(e^{-aO^1_t}X_t-Y_t)+\kappa(b_2O^2_t-\frac{b_1}{a}(1-e^{aO^1_t}))\Big]dt+b_2dB^{H_2}_t.
\end{array}\right.
\end{equation}

This system has a unique stationary solution $(\Bar{X}_\kappa,\Bar{Y}_\kappa)$, which is pathwise globally asymptotically stable with
\begin{align}
  \Big(\Bar{X}_\kappa(\theta_t\omega),\Bar{Y}_\kappa(\theta_t\omega)\Big)\rightarrow\Big(\Bar{W}(\theta_t\omega)\cdot e^{aO^1_t(\omega)}+\frac{b_1}{a}(e^{aO^1_t(\omega)}-1),\Bar{W}(\theta_t\omega)+b_2O^2_t(\omega)\Big), \notag
\end{align}
 as $\kappa\rightarrow\infty$.\\

{\textbf{Declaration of competing interest}

The authors declare that they have no known competing financial interests or personal relationships that could appeared to influence the work reported in this paper.}\\

{\textbf{Data availability}

No data was used for the research described in the article.}

\section*{Acknowledgements} {This work is supported in part by the NSFC Grant Nos. 12171084,  the fundamental Research
Funds for the Central Universities No. RF1028623037 and
the Jiangsu Center for Collaborative Innovation in Geographical Information Resource and Applications.}

\bibliographystyle{abbrv}

\end{document}